\documentclass[11pt,twoside]{article}
\usepackage{a4}
\usepackage{amssymb,amsmath,amsthm,latexsym}
\usepackage{amsfonts}
  \usepackage{amsfonts}
\usepackage{graphicx}

\newtheorem{theorem}{Theorem}[section]

\newtheorem{corollary}[theorem] {Corollary}
\newtheorem{definition}[theorem]{Definition}

\newtheorem{question}[theorem]{Question}
\newtheorem{lemma} [theorem]{Lemma}

\vspace{5cm}

\begin{document}
  
  \label{'ubf'}  
\setcounter{page}{1}                                 

\markboth {\hspace*{-9mm} \centerline{\footnotesize \sc
   A concept of largness of monochromatic sums and products in large integral domain  }
                 }
                { \centerline                           {\footnotesize \sc  
         Pintu Debnath                                                 } \hspace*{-9mm}              
               }

\vspace*{-2cm}

\begin{center}
{ 
       {\Large \textbf { \sc A concept of largness of monochromatic sums and products in large integral domain
                               }
       }
\\

\medskip

\author{Debnath, Pintu}
{\sc Pintu Debnath }\\
{\footnotesize Department of Mathematics,
			Basirhat College,
			Basirhat-743412, North 24th Parganas, West Bengal, India.}\\

{\footnotesize e-mail: {\it pintumath1989@gmail.com}}
}
\end{center}

\thispagestyle{empty}

\hrulefill

\begin{abstract}
 An infinite integral domain $R$ is called a large ideal domain (LID) if every nontrivial ideal of $R$ has finite index in $R$. Recently,  N. Hindman and D. Strauss have established a refinement of Moreira's theorem for the set of natural numbers and infinite fields. In this article, we prove the same result of N. Hindman and D. Strauss for large ideal domains (LID) and a polynomial extension.

\end{abstract}

\textbf{MSC 2020:} 05D10, 22A15, 54D35.

\textbf{Keywords:}  Piecewise syndetic set, $IP$-set, Large ideal domain, Moreira's 
theorem, Polynomial van der Waerden’s theorem, Algebra of the  Stone-\v{C}ech compactifications of discrete semigroups.

\section{Introduction}

We begin this introductory section with a well-known and still open
problem posed by N.~Hindman, I.~Leader, and D.~Strauss
\cite[Question~3]{HLS}.

\begin{question}
If the natural numbers are finitely colored, must there exist
$x,y\in\mathbb{N}$ such that
\[
x,\; y,\; x+y,\; \text{and } xy
\]
are all monochromatic?
\end{question}

In \cite{M17}, J.~Moreira established a significant partial answer to
this question.

\begin{theorem}[Moreira]\label{Moreira's Theorem}
\cite[Corollary~1.5]{M17}
For any finite coloring of $\mathbb{N}$, there exist infinitely many
pairs $x,y\in\mathbb{N}$ such that the set
\[
\{x,xy,x+y\}
\]
is monochromatic.
\end{theorem}

\begin{definition}[Piecewise syndetic]
Let $(S,+)$ be a commutative semigroup and let $A\subseteq S$.  
The set $A$ is called \emph{piecewise syndetic} if and only if there
exists $G\in\mathcal{P}_f(S)$ such that for every
$F\in\mathcal{P}_f(S)$ there exists $x\in S$ satisfying
\[
F+x \subseteq \bigcup_{t\in G}(-t+A).
\]
\end{definition}

In \cite[Corollary~1.11]{HS24}, N.~Hindman and D.~Strauss obtained a
refinement of Moreira’s theorem, stated as follows.

\begin{theorem}\label{Moreira by HS}
Let $S=\mathbb{N}$ or an infinite field, let $r\in\mathbb{N}$, and let
$S=\bigcup_{i=1}^{r} C_i$ be a finite coloring of $S$. Then there exist
$i\in\{1,2,\ldots,r\}$ and infinitely many $y\in S$ such that the set
\[
\left\{
x\in S :
\{x,xy,x+y\}\subseteq C_i
\right\}
\]
is piecewise syndetic.
\end{theorem}

Examples of large ideal domains (LIDs) include all fields, the ring
$\mathbb{Z}$ (and more generally the ring of integers of any number
field), and the polynomial ring $\mathbb{F}[x]$ over a finite field
$\mathbb{F}$. J.~Moreira proved the following result for LIDs.

\begin{theorem}\cite[Theorem~7.5]{M17}
Let $R$ be a large ideal domain, let $s\in\mathbb{N}$, and for each
$i=1,2,\ldots,s$ let $F_i$ be a finite set of functions
$R^i\to R$ such that for all $f\in F_i$ and all
$x_1,\ldots,x_{i-1}\in R$, the function
\[
x\mapsto f(x_1,x_2,\ldots,x_{i-1},x)
\]
is a polynomial with zero constant term. Then for any finite coloring
of $R$, there exist a color class $C\subseteq R$ and infinitely many
$(s+1)$-tuples $x_0,x_1,\ldots,x_s\in R$ such that
\[
\{x_0x_1\cdots x_s\}
\cup
\Bigl\{
x_0\cdots x_j + f(x_{j+1},\ldots,x_i)
:
0\leq j<i\leq s,\;
f\in F_{i-j}
\Bigr\}
\subseteq C.
\]
\end{theorem}

An immediate consequence of this theorem is the following corollary.

\begin{corollary}\label{abundance polynomial integer}
Let $R$ be a large ideal domain, let $k\in\mathbb{N}$, and let
$F\subseteq xR[x]$ be finite. Then for any finite coloring of $R$, there
exist $x,y\in R$ such that
\[
\{xy\}\cup\{x+f(y):f\in F\}
\]
is monochromatic.
\end{corollary}

In this article, we prove the following polynomial analogue of
Theorem~\ref{Moreira by HS} for large ideal domains.

\begin{theorem}\label{HS type moreira LID}
Let $R$ be a large ideal domain, let $k,r\in\mathbb{N}$, and let
$F\subseteq xR[x]$ be finite. If $R=\bigcup_{i=1}^{r} C_i$ is a finite
coloring of $R$, then there exist
$i\in\{1,2,\ldots,r\}$ and infinitely many $y\in R$ such that the set
\[
\left\{
x\in R :
\{xy\}\cup\{x+f(y):f\in F\}\subseteq C_i
\right\}
\]
is piecewise syndetic.
\end{theorem}

We conclude this section with a brief review of the algebraic structure
of the Stone--Čech compactification of discrete semigroups, which will
be used in the proofs of several results in the next two sections.

Let $S$ be a discrete semigroup. The elements of $\beta S$ are
identified with ultrafilters on $S$. For $A\subseteq S$, define
\[
\overline{A}=\{p\in\beta S : A\in p\}.
\]
The collection $\{\overline{A}:A\subseteq S\}$ forms a basis for the
closed sets of $\beta S$. The semigroup operation on $S$ extends to
$\beta S$ so that $(\beta S,\cdot)$ is a compact right topological
semigroup. This means that for each $p\in\beta S$, the map
$\rho_p:\beta S\to\beta S$ defined by $\rho_p(q)=q\cdot p$ is continuous,
and $S$ is contained in the topological center of $\beta S$, i.e., for
each $x\in S$, the map $\lambda_x:\beta S\to\beta S$ given by
$\lambda_x(q)=x\cdot q$ is continuous.

A classical theorem of Ellis asserts that every compact right
topological semigroup contains an idempotent. A nonempty subset $I$ of a
semigroup $T$ is called a \emph{left ideal} if $TI\subseteq I$, a
\emph{right ideal} if $IT\subseteq I$, and a \emph{two-sided ideal} (or
simply an \emph{ideal}) if it is both a left and a right ideal. A
minimal left ideal is a left ideal containing no proper left ideal.
Minimal right ideals and the smallest ideal are defined analogously.

Every compact Hausdorff right topological semigroup $T$ has a smallest
two-sided ideal given by
\[
K(T)
=
\bigcup\{L : L \text{ is a minimal left ideal of } T\}
=
\bigcup\{R : R \text{ is a minimal right ideal of } T\}.
\]

If $L$ is a minimal left ideal and $R$ is a minimal right ideal, then
$L\cap R$ is a group and in particular contains an idempotent. For
idempotents $p,q\in T$, we write $p\leq q$ if and only if $pq=qp=p$.
An idempotent is minimal with respect to this order if and only if it
belongs to $K(T)$. Finally, for $p,q\in\beta S$ and $A\subseteq S$, we
have
\[
A\in p\cdot q
\quad\text{if and only if}\quad
\{x\in S : x^{-1}A\in q\}\in p,
\]
where $x^{-1}A=\{y\in S : x\cdot y\in A\}$. For further details, see
\cite{HS12}.

\section{  A refinement of $IP$ van der Waerden's Theorem} 

 This section is devoted primarily to the proof of
Theorem~\ref{Ip pol van der LID PS}, which will be used in the proof of
Theorem~\ref{main theorem}.

Let $(S,+)$ be a commutative semigroup and let $A\subseteq S$.
\begin{itemize}
\item \textbf{($IP$-set)}  
The set $A$ is called an \emph{$IP$-set} if and only if there exists a
sequence $\langle x_n\rangle_{n=1}^{\infty}$ in $S$ such that
\[
FS(\langle x_n\rangle_{n=1}^{\infty})\subseteq A,
\]
where
\[
FS(\langle x_n\rangle_{n=1}^{\infty})
=\left\{\sum_{n\in F} x_n : F\in\mathcal{P}_f(\mathbb{N})\right\}.
\]

\item \textbf{($IP^{\star}$-set)}  
The set $A$ is called an \emph{$IP^{\star}$-set} if it intersects every
$IP$-set.

\item \textbf{(Syndetic set)}  
The set $A$ is called \emph{syndetic} if there exists a finite subset
$F\subseteq S$ such that
\[
S\subseteq \bigcup_{s\in F} (-s+A).
\]
\end{itemize}

As a consequence of the Polynomial Hales--Jewett Theorem
\cite{BL99}, we obtain the following polynomial version of
van der Waerden’s theorem for $IP$-sets over infinite commutative rings.
We include a proof for completeness.

\begin{theorem}\label{pvw ip in LID}
Let $R$ be an infinite commutative ring and let $r\in\mathbb{N}$. For any
partition
\[
R=\bigcup_{s=1}^{r} C_s,
\]
there exists $s\in\{1,\ldots,r\}$ such that for every
$F\in\mathcal{P}_f(xR[x])$,
\[
\left\{
n\in R :
\exists a\in R \text{ such that } a+f(n)\in C_s
\text{ for all } f\in F
\right\}
\]
is an $IP^{\star}$-set in $(R,+)$. Equivalently,
\[
\left\{
n\in R :
\bigcap_{f\in F} \bigl(-f(n)+C_s\bigr)\neq\emptyset
\right\}
\]
is an $IP^{\star}$-set in $(R,+)$.
\end{theorem}

To prove Theorem~\ref{pvw ip in LID}, we recall the classical
Hales--Jewett Theorem. Let
\[
\omega=\mathbb{N}\cup\{0\}.
\]
Given a nonempty set $\mathbb{A}$ (called an \emph{alphabet}), a
\emph{finite word} over $\mathbb{A}$ is an expression
$w=a_1a_2\cdots a_n$ with $n\geq1$ and $a_i\in\mathbb{A}$. The length of
$w$ is denoted by $|w|$. Let $v$ be a symbol not belonging to
$\mathbb{A}$. A \emph{variable word} over $\mathbb{A}$ is a word over
$\mathbb{A}\cup\{v\}$ containing at least one occurrence of $v$. For a
variable word $w$ and $a\in\mathbb{A}$, we write $w(a)$ for the word
obtained by replacing each occurrence of $v$ by $a$.

\begin{theorem}[Hales--Jewett Theorem]\label{HJ}
For all $t,r\in\mathbb{N}$, there exists a number $\mathrm{HJ}(r,t)$ such
that whenever $N\geq \mathrm{HJ}(r,t)$ and $[t]^N$ is $r$-colored, there
exists a variable word $w$ such that
\[
\{w(a): a\in[t]\}
\]
is monochromatic.
\end{theorem}

The space $[t]^N$ is called a \emph{Hales--Jewett space}, and the number
$\mathrm{HJ}(r,t)$ is called the \emph{Hales--Jewett number}.

Bergelson and Leibman established a polynomial extension of the
Hales--Jewett Theorem in \cite{BL99} using methods from topological
dynamics; a combinatorial proof was later given by Walters \cite{W}. We
now introduce the notation needed to state this result.

Let $q,N\in\mathbb{N}$ and set $Q=[q]^N$. For $a\in Q$,
$\emptyset\neq\gamma\subseteq [N]$, and $1\leq x\leq q$, define
$a\oplus x\gamma$ to be the vector $b\in Q$ given by
\[
b_i=
\begin{cases}
x, & i\in\gamma,\\
a_i, & \text{otherwise}.
\end{cases}
\]

In the Polynomial Hales--Jewett Theorem, an element
$a\in Q$ is written as
\[
a=\langle \vec{a}_1,\vec{a}_2,\ldots,\vec{a}_d\rangle,
\]
where $\vec{a}_j\in[q]^{N^j}$ for $j=1,\ldots,d$. Writing
$\vec{a}_j=\langle a_{j,\vec{i}}\rangle_{\vec{i}\in N^j}$, define
\[
a\oplus x_1\gamma\oplus x_2(\gamma\times\gamma)\oplus\cdots\oplus
x_d\gamma^d=b,
\]
where $b=\langle\vec{b}_1,\ldots,\vec{b}_d\rangle$ and
\[
b_{j,\vec{i}}=
\begin{cases}
x_j, & \vec{i}\in\gamma^j,\\
a_{j,\vec{i}}, & \text{otherwise}.
\end{cases}
\]

\begin{theorem}[Polynomial Hales--Jewett Theorem]\label{PHJ}
For all $q,k,d\in\mathbb{N}$ there exists $N(q,k,d)\in\mathbb{N}$ such
that whenever
\[
Q=[q]^N\times[q]^{N^2}\times\cdots\times[q]^{N^d}
\]
is $k$-colored, there exist $a\in Q$ and
$\gamma\subseteq [N]$ such that the set
\[
\left\{
a\oplus x_1\gamma\oplus x_2(\gamma\times\gamma)\oplus\cdots\oplus
x_d\gamma^d : 1\leq x_j\leq q
\right\}
\]
is monochromatic.
\end{theorem}

\begin{proof}[Proof of Theorem~\ref{pvw ip in LID}]
Let $F=\{f_i : i=1,\ldots,n\}$ and let $d_i=\deg(f_i)$.
Set $d=\max\{d_i : 1\leq i\leq n\}$. Write
\[
f_i(x)=\sum_{j=1}^{d} a_j^i x^j,
\]
where $a_j^i=0$ for $j>d_i$. Let $q=nd$ and let
$N=\mathrm{PHJ}(q,r,d)$.

Let
\[
A=\{a_j^i : 1\leq i\leq n,\ 1\leq j\leq d\}.
\]
As observed by Walters \cite{W}, only the cardinality of the alphabet is
relevant in Theorem~\ref{PHJ}, and since $|A|\leq q$, the theorem applies
to
\[
Q=A^N\times A^{N^2}\times\cdots\times A^{N^d}.
\]

Define $\sigma:Q\to R$ by
\[
\sigma(u)
=r+\sum_{j=1}^{d}\sum_{\vec{i}\in N^j} u_{j,\vec{i}}\,y_{\vec{i}},
\]
where $u=\langle \vec{u}_1,\ldots,\vec{u}_d\rangle$ and
$\vec{u}_j=\langle u_{j,\vec{i}}\rangle_{\vec{i}\in N^j}$.

Let $\phi:R\to\{1,2,\ldots,r\}$ be the coloring map. Then
$\phi\circ\sigma$ is an $r$-coloring of $Q$. By
Theorem~\ref{PHJ}, there exist $u\in Q$ and
$\gamma\subseteq [N]$ such that $\phi\circ\sigma$ is constant on
\[
\left\{
u\oplus x_1\gamma\oplus x_2\gamma^2\oplus\cdots\oplus x_d\gamma^d :
x_j\in A
\right\}.
\]

Let
\[
s=\sum_{j=1}^{d}\sum_{\vec{i}\in N^j\setminus\gamma^j}
u_{j,\vec{i}}\,y_{\vec{i}}.
\]
Then for $x_j=a_j^i$ we obtain
\[
\begin{aligned}
\sigma\bigl(
u\oplus x_1\gamma\oplus\cdots\oplus x_d\gamma^d
\bigr)
&= r+s+\sum_{j=1}^{d} a_j^i \sum_{\vec{i}\in\gamma^j} y_{\vec{i}} \\
&= r+s+\sum_{j=1}^{d} a_j^i (y_\gamma)^j \\
&= r+s+f_i(y_\gamma),
\end{aligned}
\]
which completes the proof.
\end{proof}

  Using the algebraic structure of the Stone--Čech compactification of
discrete semigroups \cite[Corollary~3.8]{H01}, N.~Hindman proved the
following strong form of the $IP$ polynomial van der Waerden theorem.

\begin{theorem}[Hindman]
Let $A$ be a piecewise syndetic subset of $(\mathbb{N},+)$ and let
$F\in\mathcal{P}_f(x\mathbb{Z}[x])$. Then the set
\[
\left\{
n\in\mathbb{Z} :
\bigcap_{f\in F}\bigl(-f(n)+A\bigr)
\text{ is piecewise syndetic in } (\mathbb{N},+)
\right\}
\]
is an $IP^{\star}$-set in $(\mathbb{Z},+)$.
\end{theorem}

We now prove the main result of this section.

\begin{theorem}\label{Ip pol van der LID PS}
Let $R$ be an infinite commutative ring. If $A$ is piecewise syndetic in
$(R,+)$ and $F\in\mathcal{P}_f(xR[x])$, then the set
\[
\left\{
n\in R :
\bigcap_{f\in F}\bigl(-f(n)+A\bigr)
\text{ is piecewise syndetic in } (R,+)
\right\}
\]
is an $IP^{\star}$-set in $(R,+)$.
\end{theorem}

We first recall a theorem of J.~Moreira
\cite[Theorem~2.23]{M16}, which motivates the proof of
Theorem~\ref{Ip pol van der LID PS}.

\begin{theorem}\label{Moreira PS}
Let $(S,+)$ be a countable commutative semigroup, and let $\mathcal{P}$
be a collection of finite subsets of $S$. Then the following statements
are equivalent:
\begin{itemize}
\item[(1)] For every finite coloring
$S=\bigcup_{i=1}^{r} C_i$, there exist
$i\in\{1,2,\ldots,r\}$, $P\in\mathcal{P}$, and $s\in S$ such that
$s+P\subseteq C_i$.

\item[(2)] For every piecewise syndetic set $A\subseteq S$, there exists
$P\in\mathcal{P}$ such that the set
\[
\{s\in S : s+P\subseteq A\}
\]
is piecewise syndetic.
\end{itemize}
\end{theorem}

Since we are not restricted to countable commutative semigroups, we now
establish an analogue of Theorem~\ref{Moreira PS} for arbitrary
commutative semigroups. We begin with the following lemma from
\cite[Theorem~4.39]{HS12}.

\begin{lemma}\label{PS Algeb Cha}
Let $(S,+)$ be a semigroup and let $p\in\beta S$. The following
statements are equivalent:
\begin{itemize}
\item[(a)] $p\in K(\beta S)$.
\item[(b)] For every $A\in p$, the set
\[
\{x\in S : -x+A\in p\}
\]
is syndetic.
\item[(c)] For every $q\in\beta S$, we have $p\in\beta S+q+p$.
\end{itemize}
\end{lemma}

\begin{theorem}\label{abun piecewise}
Let $S$ be a commutative semigroup and let $\mathcal{P}$ be a collection
of finite subsets of $S$. Then the following statements are equivalent:
\begin{itemize}
\item[(1)] For every finite coloring
$S=\bigcup_{i=1}^{r} C_i$, there exist
$i\in\{1,2,\ldots,r\}$, $P\in\mathcal{P}$, and $s\in S$ such that
$s+P\subseteq C_i$.

\item[(2)] For every piecewise syndetic set $A\subseteq S$, there exists
$P\in\mathcal{P}$ such that the set
\[
\{s\in S : s+P\subseteq A\}
\]
is piecewise syndetic.
\end{itemize}
\end{theorem}

\begin{proof}
The implication $(2)\Rightarrow(1)$ is immediate.

To prove $(1)\Rightarrow(2)$, let
$p\in K(\beta S)\cap\overline{A}$. By
Lemma~\ref{PS Algeb Cha}, the set
\[
B=\{x\in S : -x+A\in p\}
\]
is syndetic. Hence there exist
$x_1,\ldots,x_\ell\in S$ such that
\[
S=\bigcup_{i=1}^{\ell} (-x_i+B).
\]
By~(1), there exist $P\in\mathcal{P}$ and $h\in S$ such that
\[
h+P\subseteq -x_i+B
\]
for some $i\in\{1,2,\ldots,\ell\}$, which implies
$x_i+h+P\subseteq B$.

Let
\[
C=\bigcap_{t\in x_i+h+P} (-t+A).
\]
Then $C\in p$, and hence $C$ is piecewise syndetic. Since $S$ is
commutative,
\[
x_i+h+C\subseteq \{s\in S : s+P\subseteq A\}.
\]
Therefore, $\{s\in S : s+P\subseteq A\}$ is piecewise syndetic.
\end{proof}

\begin{proof}[Proof of Theorem~\ref{Ip pol van der LID PS}]
Let $A\subseteq R$ be a piecewise syndetic set. To prove the theorem, it
suffices to show that for any sequence
$\langle x_n\rangle_{n=1}^{\infty}$ in $R$, the set
\[
\left\{
a\in R :
\exists K\in\mathcal{P}_f(\mathbb{N}) \text{ such that }
a+f\!\left(\sum_{t\in K} x_t\right)\in A
\text{ for all } f\in F
\right\}
\]
is piecewise syndetic.

Define the family of finite subsets of $R$ by
\[
\mathcal{P}
=
\left\{
\left\{
f\!\left(\sum_{t\in H} x_t\right) : f\in F
\right\}
:
H\in\mathcal{P}_f(\mathbb{N})
\right\}.
\]

Let $r\in\mathbb{N}$ and suppose that
\[
R=\bigcup_{i=1}^{r} A_i
\]
is an arbitrary finite partition of $R$. By
Theorem~\ref{pvw ip in LID}, there exist
$H\in\mathcal{P}_f(\mathbb{N})$,
$i\in\{1,2,\ldots,r\}$, and $s\in R$ such that
\[
s+P\subseteq A_i,
\quad\text{where}\quad
P=\left\{
f\!\left(\sum_{t\in H} x_t\right) : f\in F
\right\}.
\]

Thus the family $\mathcal{P}$ satisfies condition~(1) of
Theorem~\ref{abun piecewise}. Since $A$ is piecewise syndetic,
condition~(2) of Theorem~\ref{abun piecewise} implies that there exists
$Q\in\mathcal{P}$ such that the set
\[
\left\{
a\in R : a+Q\subseteq A
\right\}
\]
is piecewise syndetic.

Because $Q\in\mathcal{P}$, there exists
$K\in\mathcal{P}_f(\mathbb{N})$ such that
\[
Q=
\left\{
f\!\left(\sum_{t\in K} x_t\right) : f\in F
\right\}.
\]
This completes the proof.
\end{proof}

\section{A refinement of  Moreira's theorem for LID}

The main result of this section is Theorem~\ref{main theorem}. As an
application, we derive Theorem~\ref{HS type moreira LID}. Before proving
the main result, we introduce several technical lemmas.

A \emph{semiring} is a triple $(S,+,\cdot)$ such that $(S,+)$ is a
commutative semigroup, $(S,\cdot)$ is a semigroup, and for all
$a,b,c\in S$,
\[
a(b+c)=ab+ac \quad \text{and} \quad (b+c)a=ba+ca.
\]

\begin{lemma}\label{distributive}\cite[Lemma~2.4]{HS24}
Let $(S,+,\cdot)$ be an infinite semiring. For all $x\in S$ and all
$p,q\in\beta S$,
\[
x(p+q)=xp+xq
\quad\text{and}\quad
(p+q)x=px+qx.
\]
\end{lemma}

\begin{lemma}\label{cancelative}\cite[Lemma~8.1]{HS12}
Suppose that $S$ is a discrete semigroup. If $s$ is a left cancellable
element of $S$, then $s$ is also left cancellable in $\beta S$. The
analogous statement holds for right cancellable elements.
\end{lemma}

Let $(S,+)$ be a commutative group and let $H$ be a subgroup of $S$. By
\cite[Theorem~1.16]{HS19}, $H$ is an $IP^{\star}$-set in $(S,+)$ if and
only if $H$ is piecewise syndetic in $S$. As a consequence,
$x\mathbb{Z}[x]$ is not piecewise syndetic in $(\mathbb{Z}[x],+)$, since
the $IP$-set $\mathbb{N}$ in $(\mathbb{Z}[x],+)$ has empty intersection
with $x\mathbb{Z}[x]$. In contrast, for LIDs we obtain the following
result.

\begin{lemma}\label{dilation ps}
Let $(R,+,\cdot)$ be a LID. If $A\subseteq R$ is piecewise syndetic in
$(R,+)$ and $r\in R\setminus\{0\}$, then $rA$ is piecewise syndetic in
$(R,+)$.
\end{lemma}

\begin{proof}
This follows from \cite[Theorem~7.4]{M17}.
\end{proof}

We also obtain the following analogue of \cite[Lemma~1.7]{HS24}.

\begin{lemma}\label{div ps}
Let $(R,+,\cdot)$ be a LID. Suppose that $A\subseteq R$ is piecewise
syndetic in $(R,+)$ and that $A\subseteq yR$ for some
$y\in R\setminus\{0\}$. Then $A/y$ is piecewise syndetic in $(R,+)$.
\end{lemma}

\begin{proof}
Pick $x\in \overline{A}\cap K(\beta R)$. Since
$x\in \overline{yR}=y\beta R$, choose $z\in\beta R$ such that $x=yz$.
Let $q\in K(\beta R)$. Then $yq\in\beta R$, so choose $u\in\beta R$ such
that
\[
yz=u+yq+yz.
\]
Since $u\in \overline{yR}=y\beta R$, there exists $w\in\beta R$ with
$u=yw$. Hence
\[
yz=yw+yq+yz.
\]
By Lemma~\ref{distributive}, this implies
\[
yz=y(w+q+z).
\]
Since $R$ is an integral domain, Lemma~\ref{cancelative} yields
\[
z=w+q+z\in K(\beta R)\cap \overline{A/y}.
\]
\end{proof}

\begin{definition}[\textbf{Uniqueness of finite products}]
Let $(S,\cdot)$ be a semigroup, let $m\in\mathbb{N}$, and let
$\langle y_t\rangle_{t=1}^{m}$ be a sequence in $S$. The sequence
satisfies \emph{uniqueness of finite products} if, whenever
$H,K\in\mathcal{P}_f(\mathbb{N})$ with $H\neq K$, we have
\[
\prod_{t\in H} y_t \neq \prod_{t\in K} y_t.
\]
An infinite sequence $\langle y_t\rangle_{t=1}^{\infty}$ satisfies
uniqueness of finite products if the same condition holds for all finite
$H,K\subseteq\mathbb{N}$.
\end{definition}

\begin{lemma}\label{UFP LID}
Let $(R,+,\cdot)$ be an integral domain with additive identity $0$ and
multiplicative identity $1$. Let $m\in\mathbb{N}$ and let
$\langle y_t\rangle_{t=1}^{m}$ be a sequence such that
\[
FP(\langle y_t\rangle_{t=1}^{m})\subseteq R\setminus\{0,1\}
\]
and satisfying uniqueness of finite products. If $A$ is an infinite
subset of $R$, then there exists $y_{m+1}\in A$ such that
\[
FP(\langle y_t\rangle_{t=1}^{m+1})\subseteq A\setminus\{0,1\},
\]
and $\langle y_t\rangle_{t=1}^{m+1}$ satisfies uniqueness of finite
products.
\end{lemma}

\begin{proof}
Let $B=FP(\langle y_t\rangle_{t=1}^{m})$ and define
\[
C=\left\{
x\in R\setminus\{0,1\} :
(\exists \alpha,\beta\in B\cup\{1\})(x\alpha=\beta)
\right\}.
\]
We claim that $C$ is finite. Suppose otherwise that $C$ is infinite and
let $\langle x_i\rangle_{i=1}^{\infty}$ be a sequence of distinct
elements of $C$. For each $i$, there exist
$\alpha_i,\beta_i\in B\cup\{1\}$ such that $x_i\alpha_i=\beta_i$.

Since $B\cup\{1\}$ is finite, there exists an infinite subset
$N_1\subseteq\mathbb{N}$ and $\beta\in B\cup\{1\}$ such that
$x_i\alpha_i=\beta$ for all $i\in N_1$. Similarly, there exist an
infinite subset $N_2\subseteq N_1$ and $\alpha\in B\cup\{1\}$ such that
$x_i\alpha=\beta$ for all $i\in N_2$. For $i,j\in N_2$ we have
$x_i\alpha=x_j\alpha$, and since $R$ is an integral domain, this implies
$x_i=x_j$, a contradiction. Hence $C$ is finite.

Choosing $y_{m+1}\in A\setminus C$ completes the proof.
\end{proof}

Since the additive group $(R,+)$ is infinite, every $IP^{\star}$-set in
$(R,+)$ is infinite. We are now ready to prove the main theorem of this
section.

\begin{theorem}\label{main theorem}
Let $R$ be a LID, let $r\in\mathbb{N}$, and suppose that
\[
R=\bigcup_{i=1}^{r} C_i .
\]
Then there exist $i\in\{1,2,\ldots,r\}$, an injective sequence
$\langle z_n\rangle_{n=1}^{\infty}$ in $S$, and a sequence
$\langle E_n\rangle_{n=1}^{\infty}$ of piecewise syndetic subsets of
$(R,+)$ such that, for each $n\in\mathbb{N}$,
\[
E_n \subseteq Sz_n,
\]
and whenever $w\in E_n$ and $xz_n=w$, we have
\[
\{xz_n,\, x+f(z_n): f\in F\}\subseteq C_i,
\]
where $F$ is a finite subset of $xR[x]$.
\end{theorem}

\begin{proof}
Choose $t_0\in\{1,2,\ldots,r\}$ such that $C_{t_0}$ is piecewise syndetic
in $(R,+)$. By Theorem~\ref{Ip pol van der LID PS}, pick
$y_1\in S\setminus\{0,1\}$ such that
\[
\bigcap_{f\in F}\bigl(B_0-f(y_1)\bigr)
\]
is piecewise syndetic, and set
\[
D_1=\bigcap_{f\in F}\bigl(B_0-f(y_1)\bigr).
\]
By Lemma~\ref{dilation ps}, $y_1D_1$ is piecewise syndetic. Since
\[
y_1D_1=\bigcup_{i=1}^{r}\bigl(y_1D_1\cap C_i\bigr),
\]
there exists $t_1\in\{1,2,\ldots,r\}$ such that
$y_1D_1\cap C_{t_1}$ is piecewise syndetic. Let
\[
B_1=y_1D_1\cap C_{t_1}.
\]

Let $k\in\mathbb{N}$ and assume that we have chosen sequences
$\langle y_j\rangle_{j=1}^{k}$,
$\langle B_j\rangle_{j=0}^{k}$,
$\langle t_j\rangle_{j=0}^{k}$,
and $\langle D_j\rangle_{j=1}^{k}$ satisfying the following induction
hypotheses.

\begin{itemize}
\item[(1)] $y_j\in S$ for $j=1,\ldots,k$, and
\[
FP(\langle y_t\rangle_{t=1}^{k})\subseteq S\setminus\{0,1\}
\]
satisfies uniqueness of finite products.

\item[(2)] $D_j$ is piecewise syndetic for $j=1,\ldots,k$.

\item[(3)] $t_j\in\{1,2,\ldots,r\}$ for $j=1,\ldots,k$.

\item[(4)] $B_j$ is piecewise syndetic for $j=1,\ldots,k$.

\item[(5)] $B_j\subseteq C_{t_j}$ for $j=1,\ldots,k$.

\item[(6)] $B_j\subseteq y_jD_j$ for $j=1,\ldots,k$.

\item[(7)] For $j<m$ in $\{0,1,\ldots,k\}$,
\[
B_m\subseteq y_m y_{m-1}\cdots y_{j+1}B_j.
\]

\item[(8)] For $m=1,\ldots,k$,
\[
D_m\subseteq B_{m-1}\cap (B_{m-1}-y_m),
\]
and if $m>1$, then
\[
D_m\subseteq
\bigcap_{j=1}^{m-1}\bigcap_{f\in F}
\left(
B_{m-1}-y_{m-1}\cdots y_j
f\bigl(y_{m-1}\cdots y_jy_m\bigr)
\right).
\]
\end{itemize}

All hypotheses hold for $k=1$.

For $j=1,\ldots,k$, let $u_j=y_k y_{k-1}\cdots y_j$. By
Theorem~\ref{Ip pol van der LID PS},
\[
A=\left\{
y\in S:
\bigcap_{j=1}^{k}\bigcap_{f\in F}
\bigl(B_k-u_jf(u_jy)\bigr)
\text{ is piecewise syndetic}
\right\}
\]
is an $IP^\ast$-set in $(R,+)$. Choose $y_{k+1}\in A$ such that
\[
FP(\langle y_t\rangle_{t=1}^{k+1})\subseteq S\setminus\{0,1\}
\]
and uniqueness of finite products holds. Define
\[
D_{k+1}=
\bigcap_{j=1}^{k}\bigcap_{f\in F}
\bigl(B_k-u_jf(u_jy_{k+1})\bigr).
\]

By Lemma~\ref{dilation ps}, $y_{k+1}D_{k+1}$ is piecewise syndetic.
Choose $t_{k+1}\in\{1,2,\ldots,r\}$ such that
$y_{k+1}D_{k+1}\cap C_{t_{k+1}}$ is piecewise syndetic, and set
\[
B_{k+1}=y_{k+1}D_{k+1}\cap C_{t_{k+1}}.
\]

The induction is complete.

Choose $i\in\{1,2,\ldots,r\}$ such that
\[
G=\{k\in\mathbb{N}: t_k=i\}
\]
is infinite. Let $\langle k(n)\rangle_{n=0}^{\infty}$ be an increasing
sequence in $G$, and define
\[
z_n=y_{k(n)}y_{k(n)-1}\cdots y_{k(n-1)+1}.
\]
Then $\langle z_n\rangle_{n=1}^{\infty}$ is injective. For each
$n\in\mathbb{N}$, let $E_n=B_{k(n)}$. Each $E_n$ is piecewise syndetic and
\[
E_n\subseteq z_n S.
\]

Let $w\in E_n$ and suppose $xz_n=w$. Then $xz_n\in C_i$. It remains to
show that $x+f(z_n)\in C_i$ for all $f\in F$. We compute
\[
\begin{aligned}
z_n(x+f(z_n))
&= w+z_nf(z_n) \\
&\in B_{k(n)}+z_nf(z_n) \\
&\subseteq y_{k(n)}D_{k(n)}+z_nf(z_n) \\
&\subseteq y_{k(n)}B_{k(n)-1}.
\end{aligned}
\]
Hence $x+f(z_n)\in B_{k(n)-1}\subseteq C_i$.
\end{proof}

\begin{proof}[Proof of Theorem~\ref{HS type moreira LID}]
Let $i$, $\langle z_n\rangle_{n=1}^{\infty}$, and
$\langle E_n\rangle_{n=1}^{\infty}$ be as given by
Theorem~\ref{main theorem}. For $n\in\mathbb{N}$, let $y=z_n$. Then
\[
E_n y^{-1}\subseteq
\left\{
x\in\mathbb{N}:
\{xy,x+f(y):f\in F\}\subseteq C_i
\right\}.
\]
By Lemma~\ref{div ps}, $E_n y^{-1}$ is piecewise syndetic.
\end{proof}

\bibliographystyle{plain}

\end{document}